\documentclass[preprint,authoryear,12pt]{elsarticle}

\newtheorem{theorem}{Theorem}
\newproof{proof}{Proof}
\newdefinition{remark}{Remark}

\usepackage{amssymb}
\usepackage{amsmath}

\journal{Mathematical Social Sciences}

\begin{document}

\begin{frontmatter}



\title{Scatter and regularity imply Benford's Law... and more}


\author{N. Gauvrit\corref{cor1}}
\ead{adems@free.fr}
\address{Laboratoire Andr\'e Revuz, University Paris VII, France}

\author{J.-P. Delahaye}
\ead{jpd@lifl.fr}
\address{Laboratoire d'Informatique Fondamentale, USTL, Lille, France}

\cortext[cor1]{Corresponding author}

\begin{abstract}
A random variable (r.v.) X is said to follow Benford's law if $\log(X)$ is
uniform mod 1. Many experimental data sets prove to follow an approximate
version of it, and so do many mathematical series and continuous random
variables. This phenomenon received some interest, and several explanations
have been put forward. Most of them focus on specific data, depending on
strong assumptions, often linked with the log function.

Some authors hinted - implicitly - that the two most important
characteristics of a random variable when it comes to Benford are regularity
and scatter.

In a first part, we prove two theorems, making up a formal version of this
intuition: scattered and regular r.v.'s do approximately follow Benford's
law. The proofs only need simple mathematical tools, making the analysis
easy. Previous explanations thus become corollaries of a more general and
simpler one.

These results suggest that Benford's law does not depend on properties
linked with the log function. We thus propose and test a general version of
the Benford's law. The success of these tests may be viewed as an \emph{a
posteriori} validation of the analysis formulated in the first part.

\end{abstract}


\end{frontmatter}




\section{Introduction}

First noticed by \citet{NEW81}, and again later by \citet{BEN38}, the
so-called Benford's law states that a sequence of "random" numbers should be
such that their logarithms are uniform mod 1. As a consequence, the
first non-zero digit of a sequence of "random" numbers is $d$ with
probability $\log \left( 1+\frac{1}{d}\right) $, an unexpectedly non-uniform
probability law. $\log $ here stands for the base 10 logarithm, but an easy
generalisation follows: a random variable (r.v.) conforms to \emph{base }$b$
Benford's law if its base $b$ logarithm $\log _{b}\left( X\right) $ is
uniform mod 1. \citet{LOL08} recently proved that no r.v. follows
base $b$ Benford's law for all $b.$

Many experimental data roughly conform to Benford's law (most of which no
more than roughly). However, the vast majority of real data sets that have
been tested do not fit this law at all. For instance, \citet{SCO01}
reported that only 12.6\% of 230 real data sets passed the test for
Benford's law. In his seminal paper, \citet{BEN38} tested 20 data sets
including lakes areas, length of rivers, populations, etc., half of which
did not conform to Benford's law.

The same is true of mathematical sequences or continuous r.v.'s. For
example, binomial arrays ${n \choose k},$ with $n\geq 1,$ $k\in \left\{
0,...,n\right\} ,$ tend toward Benford's law \citep{DIA77}, whereas
simple sequences such as $\left( 10^{n}\right) _{n\in \mathbb{N}}$ obviously
don't.

In spite of all this, Benford's law is actually used in the so-called
"digital analysis" to detect anomalies in pricing \citep{SEH05} or
frauds, for instance in accounting reports \citep{DRA00} or campaign finance \citep{CHO07}. Faked data indeed
usually depart from Benford's law more than real ones \citep{HIL88}. However,
\citet{HAL08} advise caution, arguing that real data do not always fit
the law.

Many explanations have been put forward to elucidate the appearance of
Benford's law on natural or mathematical data. Some authors focus on
particular random variables \citep{ENG03}, sequence \citep{JOL05}, 
real data \citep{BUR91}, or orbits of
dynamical systems \citep{BER04}. As a rule, other explanations
assume special properties of the data. \citet{HIL95b} or \citet{PIN61} shows
that scale invariance implies Benford's law. Base invariance is an other
sufficient condition \citep{HIL95a}. Mixtures of uniform distributions
\citep{JAN04} also conform to Benford's law, and so do the
limits of some random processes \citep{SHU08}. Multiplicative
processes have been mentioned as well \citep{PIE01}. Each of
these explanations accounts for some appearances of data fitting Benford's
law, but lacks generality.

While looking for a truly general explanation, some authors noticed that
data sets are more likely to fit Benford's law if they were scattered
enough. More precisely, a sequence should "cover several orders of
magnitude", as \citet{RAI76} expressed it. Of course, scatter alone is no
sufficient condition. The sequence 0.9, 9, 90, 900... indeed covers several
orders of magnitude, but is far from conforming to Benford's law. The
continuous random variables that are known to fit Benford's law usually
present some "regularity": exponential densities, normal densities, or
lognormal densities are of this kind. Invariance assumptions
(base-invariance or scale-invariance) lead to "regular" densities and so do
central limit-like theorem assumptions of mixture.

Some technical explanations may be viewed as a mathematical expression of
the idea that a random variable $X$ is more likely to conform to Benford's
law if it is regular and scattered enough. \cite[Example 4.1.]{MAR00} linked
Benford's law to Poincar\'{e}'s theorem in circular statistics, and \citet{SMI07} 
expressed it in terms of Fourier transforms and signal processing.
However, a non expert reader would hardly notice the smooth-and-scattered
implications of these developments.

Though scatter has been explicitly mentioned and regularity allusively
evoked, the idea that scatter and regularity (in a sense that will be made
clear further) may actually be a \emph{sufficient} explanation for Benford's
phenomenon related to continuous r.v.'s have never been formalized in a
simple way, to our knowledge, except in a recent article by \citet{FEW09}.
In this paper, Fewster hypothesizes that "\emph{any distribution [...] that
is reasonably smooth and covers several orders of magnitude is almost
guaranteed to obey Benford's law." }He then defines a smoothing procedure
for a r.v. $X$ based on $\left[ \pi ^{2}\left( x\right) \right] ^{\prime
\prime },$ $\pi $ being the probability density function (henceforth \emph{%
p.d.f.}) of $\log \left( X\right) ,$ and illustrates with a few eloquent
examples that under smoothness and scatter constraints, a r.v. cannot depart
much from Benford's law. However, no theorem is given that would formalise
this idea.

In the first part of this paper, we prove a theorem from which it follows
that scatter and regularity can be modelled in such a way that they, alone,
imply \emph{rough} compliance to Benford's law (again: real data usually do
not perfectly fit Benford's law, irrespective of the sample size).

It is not surprising that many data sets or random variables samples are
scattered and regular hence our explanation of Benford's phenomena
corroborates a widespread intuition. The proof of this theorem is
straightforward and requires only basic mathematical tools. Furthermore, as
we shall see, several of the existing explanations can be understood as
corollaries of ours. Our explanation encompasses more specific ones, and is
far simpler to understand and to prove.

Scatter and regularity do not presuppose any log-related properties (such as
the property of log-normality, scale-invariance, or multiplicative
properties). For this reason, if we are right, Benford's law should also
admit other versions. We set that a r.v. $X$ is $u$-Benford for a function $%
u $ if $u(X)$ is uniform mod 1. The classical Benford's law is thus
a special case of $u$-Benford's law, with $u=\log $. We test real data sets
and mathematical sequences for "$u$-Benfordness" with various $u$, and test
a second theorem echoing the first one. Most data conform to $u$-Benford's
law for different $u$, which is an argument in favour of our explanation.

\section{Scatter and regularity: a key to Benford}

The basic idea at the root of theorem 1 (below) is twofold.

First, we hypothesize that a continuous r.v. $X$ with density $f$ is almost
uniform mod 1 as soon as it is scattered and regular. More
precisely, any $f$ that is non-decreasing on $]-\infty ,a]$, and then
non-increasing on $[a,+\infty \lbrack $ (for regularity) and such that its
maximum $m=\sup (f)$ is "small" (for scatter) should correspond to a r.v. $X$
approaching uniformity mod 1. Figure 1 illustrates this idea.
\\

\includegraphics[scale=.4]{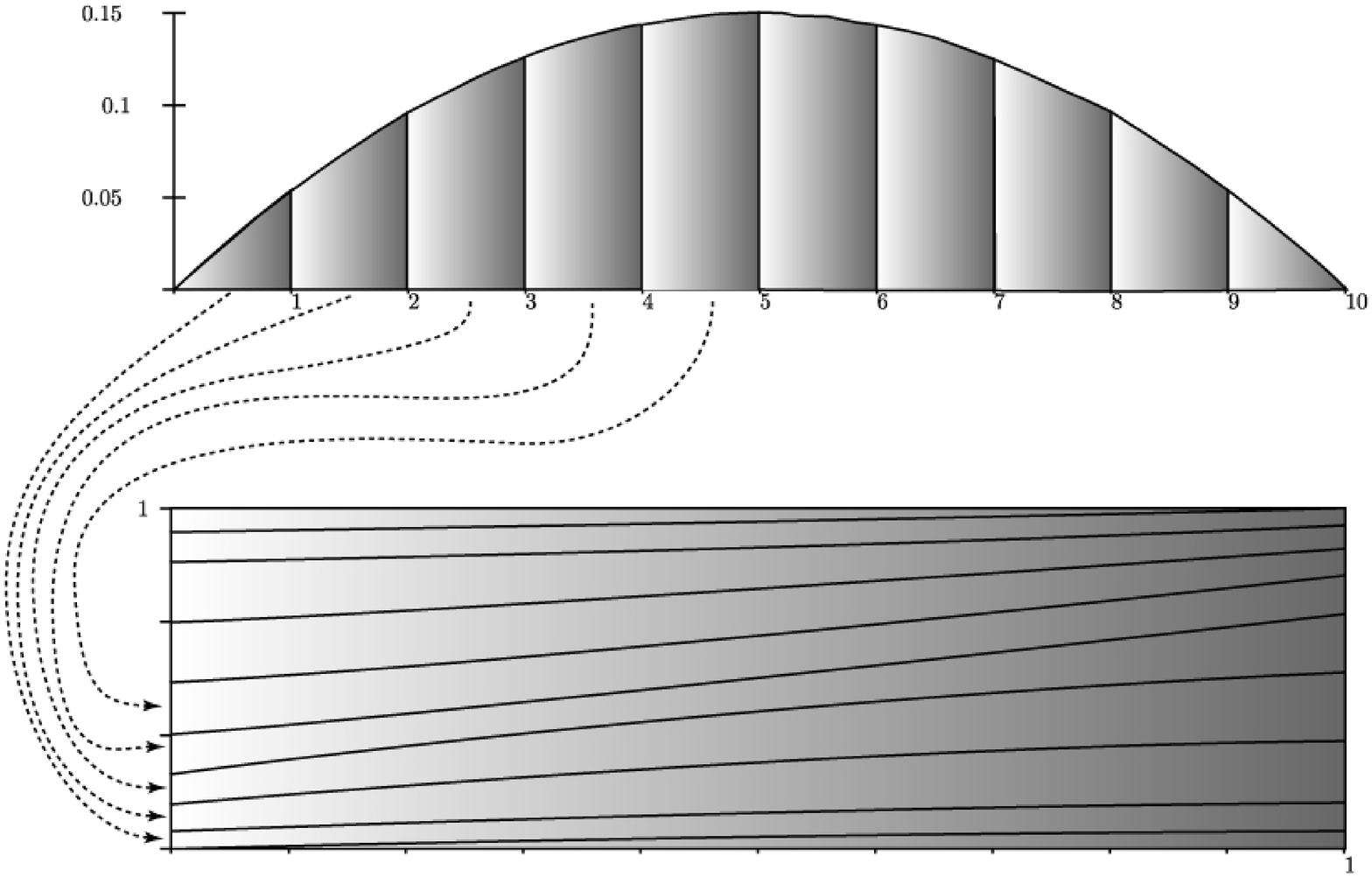} 

\begin{center}
Figure 1 --- Illustration of the idea that a regular
p.d.f. is almost bound to give rise to uniformity mod 1. The
stripes --- restrictions of the density on $\left[ n,n+1\right] $ --- of the
p.d.f. of a r.v. $X$ are stacked to form the p.d.f. of $X$ mod 1.
The slopes partly compensate, so that the resulting p.d.f. is almost
uniform. If the initial p.d.f. is linear on every $\left[ n,n+1\right] ,$
the compensation is perfect.
\end{center}

Second, note that if $X$ is scattered and regular enough, so should be $\log
\left( X\right) $. These two ideas are formalized and proved in theorem 1.

Henceforth, for any real number $x,$ $\left\lfloor x\right\rfloor $ will
denote the greatest integer not exceeding $x,$ and $\left\{ x\right\}
=x-\left\lfloor x\right\rfloor $. Any positive $x$ can be written as a
product $x=10^{\left\lfloor \log x\right\rfloor }.10^{\left\{ \log x\right\}
},$ and the Benford's law may be rephrased as the uniformity of the random
variable $\left\{ \log \left( X\right) \right\} .$

\begin{theorem}
Let $X$ be a continuous positive random variable with p.d.f. $f$ such that
 $Id.f:$ $x\longmapsto xf(x)$ conforms to the two following conditions~: 
 $\exists a>0$ such that (1) $\max (Id.f)=m=a.f(a)$ and (2) $Id.f$ is
nondecreasing on $]0,a],$ and nonincreasing on $[a,+\infty \lbrack .$ Then,
for any $z\in ]0,1],$
\begin{equation*}
\left\vert P\left( \left\{ \log X\right\} <z\right) -z\right\vert <2\ln
\left( 10\right) m.
\end{equation*}
In particular, $\left( X_{n}\right) $ being a sequence of continuous r.v.'s
with p.d.f. $f_{n}$ satisfying these conditions and such that $m_{n}=\max
\left( Id.f_{n}\right) \longrightarrow 0$, $\left\{ \log \left( X_{n}\right)
\right\} $ converges toward uniformity on $[0,1[$ in law.
\end{theorem}

\begin{proof}
We first prove that for any continuous r.v. $Y$ with density $g$ such that $%
g $ is nondecreasing on $]-\infty ,b],$ and then nonincreasing on $%
[b,+\infty \lbrack ,$ the following holds:%
\begin{equation*}
\forall z\in ]0,1],~\left\vert P\left( \left\{ Y\right\} <z\right)
-z\right\vert <2M
\end{equation*}%
where $M=g\left( b\right) =\sup \left( g\right) .$

We may suppose without loss of generality that $b\in \lbrack 0,1[.$ Let $%
z\in ]0,1[$ (the case $z=0$ is obvious). Put $I_{n,z}=[n,n+z[.$ For any
integer $n\leq -1,$%
\begin{equation*}
\frac{1}{z}\int_{I_{n,z}}g(t)dt\leq \int_{n}^{n+1}g(t)dt.
\end{equation*}%
Thus%
\begin{equation*}
\frac{1}{z}\sum_{n\leq -1}\int_{I_{n,z}}g(t)dt\leq \int_{-\infty }^{0}g(t)dt.
\end{equation*}

For any integer $n\geq 2,$%
\begin{equation*}
\frac{1}{z}\int_{I_{n,z}}g(t)dt\leq \int_{n-1+z}^{n+z}g(t)dt,
\end{equation*}%
so%
\begin{equation*}
\frac{1}{z}\sum_{n\geq 2}\int_{I,z}g(t)dt\leq \int_{1+z}^{+\infty }g(t)dt.
\end{equation*}

Moreover, $\int_{I_{0,z}}g\leq zM$ and $\int_{I,z}g\leq zM.$ Hence,%
\begin{equation*}
\frac{1}{z}\sum_{n\in \mathbb{Z}}\int_{I_{n,z}}g\leq \int_{-\infty }^{\infty
}g+2M.
\end{equation*}%
We prove in the same fashion that%
\begin{equation*}
\frac{1}{z}\sum_{n\in \mathbb{Z}}\int_{I_{n,z}}g\geq \int_{-\infty }^{\infty
}g-2M.
\end{equation*}

Since $\sum_{\mathbb{Z}}\int_{I_{n,z}}g=P\left( \left\{ Y\right\} <z\right)
, $ $z<1$ and $\int_{-\infty }^{\infty }g=1,$ the result is proved.

Now, applying this to $Y=\log \left( X\right) $ proves theorem 1.
\end{proof}

\begin{remark}
The convergence theorem is still valid if we accept $f$ to have a finite
number of monotony changes, provided this number does not exceed a
previously fixed $k$. The proof is straightforward.
\end{remark}

\begin{remark}
The assumptions made on $Id.f$ may be seen as a measure of scatter and
regularity for $X,$ adjusted for our purpose.
\end{remark}

\section{Examples}

\subsection{Type I Pareto}

A continuous r.v. $X$ is type I Pareto with parameters $\alpha $ and $x_{0}$ 
$(\alpha ,~x_{0}\in \mathbb{R}_{+}^{\ast })$ iff it admits a density function%
\begin{equation*}
f_{x_{0},\alpha }\left( x\right) =\frac{\alpha x_{0}^{\alpha }}{x^{\alpha +1}%
}\mathbb{I}_{[x_{0},+\infty \lbrack }
\end{equation*}

Besides its classical use in income and wealth modelling, type I Pareto
variables arise in hydrology and astronomy \citep[page 252]{PAO06}.

The function $Id.f=g:$ $x\longmapsto \frac{\alpha x_{0}^{\alpha }}{x^{\alpha
}}\mathbb{I}_{[x_{0},\infty \lbrack }$ is decreasing. Its maximum is%
\begin{equation*}
\sup \left( Id.f\right) =Id.f\left( x_{0}\right) =\alpha .
\end{equation*}
Therefore, $X$ is nearly Benford-like, in the extent that%
\begin{equation*}
\left\vert P(\{\log X\}<z)-z\right\vert <2\ln (10)\alpha .
\end{equation*}

\subsection{Type II Pareto}

A r.v. $X$ is type II Pareto with parameter $b>0$ iff it admits a density
function defined by%
\begin{equation*}
f_{b}\left( x\right) =\frac{b}{\left( 1+x\right) ^{b+1}}\mathbb{I}%
_{[0,+\infty \lbrack }
\end{equation*}

It arises in a so-called \emph{mixture model}, with mixing components being
gamma distributed r.v.'s sequences.

The function $Id.f_{b}=g_{b}:$ $x\longmapsto \frac{bx}{\left( 1+x\right)
^{b+1}}\mathbb{I}_{[0,+\infty \lbrack }$ is $C^{\infty }\left( \mathbb{R}%
_{+}\right) ,$ with derivative%
\begin{equation*}
g_{b}^{\prime }\left( x\right) =\frac{b\left( 1-bx\right) }{\left(
x+1\right) ^{b+2}},
\end{equation*}%
which is positive whenever $x<\frac{1}{b},$ then negative. From this result
we derive%
\begin{equation*}
\sup g_{b}=g_{b}\left( \frac{1}{b}\right) =\frac{1}{\left( 1+\frac{1}{b}%
\right) ^{1+b}}=\left( \frac{b}{1+b}\right) ^{b+1},
\end{equation*}%
since%
\begin{equation*}
\ln \left[ \left( \frac{b}{1+b}\right) ^{b+1}\right] =\left( b+1\right) %
\left[ \ln b-\ln \left( b+1\right) \right] ,
\end{equation*}%
which tends toward $-\infty $ when $b$ tends toward 0,%
\begin{equation*}
\sup g_{b}\underset{b\longrightarrow 0}{\longrightarrow }0.
\end{equation*}%
Theorem 1 applies. It follows that $X$ conform toward Benford's law when $%
b\longrightarrow 0.$

\subsection{Lognormal distributions}

A r.v. $X$ is lognormal iff $\log \left( X\right) \thicksim N\left( \mu
,\sigma ^{2}\right) $. Lognormal distributions have been related to Benford
\citep{KOS06}. It is easy to prove that whenever $\sigma \longrightarrow
\infty ,$ $X$ tends toward Benford's law. Although the proof may use
different tools, a straightforward way to do it is theorem 1.

One classical explanation of Benford's law is that many data sets are
actually built through multiplicative processes \citep{PIE01}.
Thus, data may be seen as a product of many small effects. This may be
modelled by a r.v. $X$ that may be written as%
\begin{equation*}
X=\prod\limits_{i}Y_{i},
\end{equation*}%
$Y_{i}$ being a sequence of random variables. Using the $\log $
transformation, this leads to $\log \left( X\right) =\sum \log \left(
Y_{i}\right) .$

The \emph{multiplicative central-limit theorem} therefore proves that, under
usual assumptions, $X$ is bound to be almost lognormal, with $\log \left(
X\right) \thicksim N\left( \mu ,\sigma ^{2}\right) ,$ and $\sigma
\longrightarrow \infty ,$ thus roughly conforming to Benford, as an
application of theorem 1.

\section{Generalizing Benford}

If we are right to think that Benford's law is to be understood as a
consequence of mere scatter and regularity, instead of special
characteristics linked with multiplicative, scale-invariance, or whatever
log-related properties, we should be able to state, prove, and check on real
data sets, a generalized version of the Benford's law were some function $u$
replaces the $\log .$

Indeed, our basic idea is that $X$ being scattered and regular enough
implies $\log \left( X\right) $ to be scattered and regular as well, so that 
$\log \left( X\right) $ should be almost uniform mod 1. The same
should be true of any $u\left( X\right) ,$ $u$ being a function preserving
scatter and regularity. Actually, some $u$ should even be better shots than $%
\log ,$ since $\log $ reduces scatter on $[1,+\infty \lbrack .$

First, let us set out a generalized version of theorem 1, the proof of which
is closely similar to that of theorem 1.

\begin{theorem}
Let $X$ be a r.v. taking values in a real interval $I,$ with p.d.f. $f.$ Let 
$u$ be a $C^{1}$ increasing function $I\longrightarrow \mathbb{R}$, such
that $\frac{f}{u^{\prime }}:$ $x\longmapsto \frac{f(x)}{u^{\prime }\left(
x\right) }$ conforms to the following: $\exists a>0$ such that (1) $\max
\left( \frac{f}{u^{\prime }}\right) =m=\frac{f}{u^{\prime }}\left( a\right) $
and (2) $\frac{f}{u^{\prime }}$ is non-decreasing on $]0,a],$ and
non-increasing on $[a,+\infty \lbrack \cap I.$ Then, for all $z\in \lbrack
0,1[,$%
\begin{equation*}
\left\vert P\left( \left\{ u\left( X\right) \right\} <z\right) -z\right\vert
<2m.
\end{equation*}%
In particular, if $\left( X_{n}\right) $ is a sequence of such r.v.'s with
p.d.f. $f_{n}$ and $\max \left( f_{n}/u^{\prime }\right) =m_{n}$, and $%
\lim_{+\infty }\left( m_{n}\right) =0,$ then $\left\{ u\left( X_{n}\right)
\right\} $ converges in law toward $U\left( [0,1[\right) $ when $%
n\longrightarrow \infty .$
\end{theorem}

A r.v. $X$ such that $\left\{ u\left( X\right) \right\} \thicksim U\left(
[0,1[\right) $ will be said $u$ -Benford henceforth.

\subsection{Sequence}

Although our two theorems only apply to continuous r.v.'s, the underlying
intuition that $\log $-Benford's law is only a special case (having,
however, a special interest thanks to its implication in terms of
leading-digits interpretation) of a more general law does also apply to
sequence. In this section, we experimentally test $u$-Benfordness for a few
sequences $\left( v_{n}\right) $ and a four functions $u.$

We will use six mathematical sequences. Three of them, namely $\left( \pi
n\right) _{n\in \mathbb{N}},$ prime numbers $\left( p_{n}\right) $, and $%
\left( \sqrt{n}\right) _{n\in \mathbb{N}}$ are known not to follow Benford.
The three others, $\left( n^{n}\right) _{n\in \mathbb{N}},$ $\left(
n!\right) _{n\in \mathbb{N}}$ and $\left( e^{n}\right) _{n\in \mathbb{N}}$
conform to Benford.

As for $u,$ we will focus on four cases:%
\begin{eqnarray*}
x &\longmapsto &\log \left[ \log \left( x\right) \right] \\
x &\longmapsto &\log \left( x\right) \\
x &\longmapsto &\sqrt{x} \\
x &\longmapsto &\pi x^{2}
\end{eqnarray*}

The first one increases very slowly, so we may expect that it will not work
perfectly. The second leads to the classical Benford's law. The $\pi $
coefficient of the last $u$ allows us to use integer numbers, for which $%
\left\{ x^{2}\right\} $ is nil.

The result of the experiment is given in Table 1.

\begin{center}
\begin{tabular}{|l|l|l|l|l|}
\hline
$v_{n}$ & $\log \circ \log \left( v_{n}\right) $ & $\log \left( v_{n}\right) 
$ & $\sqrt{v_{n}}$ & $\pi v_{n}^{2}$ \\ \hline\hline
$\sqrt{n}$ $\left( N=10~000\right) $ & $68.90~(.000)$ & $45.90~(.000)$ & $%
4.94~(.000)$ & $0.02~(.000)$ \\ \hline
$\pi n$ $\left( N=10~000\right) $ & $44.08~(.000)$ & $26.05~(.000)$ & $%
0.19~(1.000)$ & $0.80~(.544)$ \\ \hline
$p_{n}$ $\left( N=10~000\right) $ & $53.92~(.000)$ & $22.01~(0.000)$ & $%
0.44~(0,990)$ & $0.69~(.719)$ \\ \hline
$e^{n}$ $\left( N=1~000\right) $ & $6.91~\left( 0.000\right) $ & $%
0.76~\left( 1.000\right) $ & $0.63~\left( .815\right) $ & $0.79~\left(
.560\right) $ \\ \hline
$n!$ $\left( N=1~000\right) ^{(\ast )}$ & $7.39~\left( .000\right) $ & $%
0.58~\left( .887\right) $ & $0.61~\left( .844\right) $ & $0.90~\left(
.387\right) $ \\ \hline
$n^{n}$ $\left( N=1~000\right) ^{(\ast )}$ & $7.45$~$\left( .000\right) $ & 
$0.80~\left( .543\right) $ & $16.32~\left( .000\right) $ & $0.74~(.646)$
\\ \hline
\end{tabular}

Table 1 --- Results of the Kolmogorov-Smirnov tests applied on $\left\{
u\left( v_{n}\right) \right\} ,$ with four different functions $u$ (columns)
and six sequences (lines). Each sequence is tested through its first $N$
terms (from $n=1$ to $n=N),$ with an exception for $\log \circ \log \left(
n^{n}\right) $ and $\log \circ \log \left( n!\right) ,$ for which $n=1$ is
not considered. Each cell displays the Kolmogorov-Smirnov $z$ and the
corresponding $p$ value.
\end{center}

The sequences have been arranged according to the speed with which it
converges to $+\infty $ (and so are the functions $u)$. None of the six
sequences is $\log \circ \log $-Benford (but a faster divergent sequence
such as $\left( 10^{e^{n}}\right) $ would do). Only the last three are $\log 
$-Benford. These are the sequences going to $\infty $ faster than any
polynomial. Only one sequence $\left( n^{n}\right) $ does not satisfy $\sqrt{%
.}$-Benfordness. However, this can be understood as a pathological case,
since $\sqrt{n^{n}}$ is integer whenever $n$ is even, or is a perfect
square. Doing the same Kolmogorov-Smirnov test with odd numbers not being
perfect squares gives $z=0,45$ and $p=0,987,$ showing no discrepancy with $%
\sqrt{.}$-Benfordness for $\left( n^{n}\right) .$ All six sequences are $\pi
.^{2}$-Benford.

Putting aside the case of $\sqrt{n^{n}},$ what Table 1 reveals is that the
convergence speed of $u\left( v_{n}\right) $ completely determines the $u$%
-Benfordness of $\left( v_{n}\right) .$ More precisely, it seems that $%
\left( v_{n}\right) $ is $u$-Benford whenever $u\left( v_{n}\right) $
increases as fast as $\sqrt{n},$ and is not $u$-Benford whenever $u\left(
v_{n}\right) $ increase as slowly as $\ln \left( n\right) .$ Of course, this
rule-of-thumb is not to be taken as a theorem. Obviously enough, one can
actually decide to increase or decrease convergence speed of $u\left(
v_{n}\right) $ without changing $\left\{ u\left( v_{n}\right) \right\} ,$
adding or substracting \emph{ad hoc }integer numbers.

Nevertheless, this observation suggests that we give a closer look at
sequence $f\left( n\right) ,$ where $f$ is an increasing and concave real
function converging toward $\infty ,$ and look for a condition for $\left(
\left\{ f\left( n\right) \right\} \right) _{n}$ to converge to uniformity.
An intuitive idea is that $\left( \left\{ f\left( n\right) \right\} \right)
_{n}$ will depart from uniformity if it does not increase fast enough: we
may define brackets of integers --- namely $[f^{-1}\left( n\right)
,f^{-1}\left( n+1\right) -1[\cap \mathbb{N},$ within which $\left\lfloor
f\left( n\right) \right\rfloor $ is constant, and of course $\left\{ f\left(
n\right) \right\} $ increasing. If these brackets are "too large", the
relative height of the last considered bracket is so important that it
overcomes the first terms of the sequence $f\left( 0\right) ,...,f\left(
n\right) $ mod 1. In that case, there is no limit to the
probability distribution of $\left( \left\{ f\left( n\right) \right\}
\right) .$ The weight of the brackets should therefore be small relative to $%
f^{-1}\left( n\right) ,$ which may be written as%
\begin{equation*}
\frac{f^{-1}\left( n\right) -f^{-1}\left( n+1\right) }{f^{-1}\left( n\right) 
}\underset{\infty }{\longrightarrow }0.
\end{equation*}

Provided that $f$ is regular, this leads to%
\begin{equation*}
\frac{\left( f^{-1}\right) ^{\prime }\left( x\right) }{f^{-1}\left( x\right) 
}\underset{\infty }{\longrightarrow }0,
\end{equation*}%
or%
\begin{equation*}
\left[ \ln \left( f^{-1}\left( x\right) \right) \right] ^{\prime }\underset{%
\infty }{\longrightarrow }0.
\end{equation*}

Functions $f:x\longmapsto x^{\alpha },$ $\alpha >0$ satisfy this condition.
Any $n^{\alpha }$ should then show a uniform limit probability law, except
for pathological cases $\left( \alpha \in \mathbb{Q}\right) .$ Taking $%
\alpha =\frac{1}{\pi }$ gives (with $N=1000),$ a Kolmogorov-Smirnov $%
z=1,331, $ and a $p$-value 0.058, which means there is no significant
discrepancy from uniformity. On the other hand, the $\log $ function which
does not conform to this condition is such that $\left\{ \log \left(
n\right) \right\} $ is not uniform, confirming once again our rule-of-thumb
conjecture.

\subsection{Real data}

We test three data sets for $u$-Benfordness using a Kolmogorov-Smirnov test
for uniformity. First data set is the opening value of the Dow Jones, the
first day of each month from October 1928 to November 2007. The second and
third are country areas expressed in millions of square-km%
${{}^2}$
and the populations of the different countries, as estimated in 2008,
expressed in millions of inhabitants. The two last sequences are provided by
the CIA\footnote{%
http://www.cia.gov/library/publications/the-world-factbook/docs/rankorderguide.html%
}. Table 2 displays the results.

\begin{center}
\begin{tabular}{|l|l|l|l|l|}
\hline
& $\log \circ \log \left( v_{n}\right) $ & $\log \left( v_{n}\right) $ & $%
\sqrt{v_{n}}$ & $\pi v_{n}^{2}$ \\ \hline
Dow Jones $\left( N=950\right) $ & $5.90~(.000)$ & $5.20~(.000)$ & $%
0.75~(.635)$ & $0.44~(.992)$ \\ \hline
Area pays $(N=256)$ & $1.94~(.001)$ & $0.51~(.959)$ & $0.89~(.404)$ & $%
1,88~(.002)$ \\ \hline
Populations $\left( N=242\right) $ & $3.39~(.000)$ & $0.79~(.568)$ & $%
0.83~(.494)$ & $0.42~(.994)$ \\ \hline
\end{tabular}

Table 2 --- Results of the Kolmogorov-Smirnov tests applied on $\left\{
u\left( v_{n}\right) \right\} .$
\end{center}

This table confirms our analysis: classical Benfordness is actually less
often borne out than $\sqrt{.}$-Benfordness on these data. The last column
shows that our previous conjectured rule has exceptions: divergence speed is
not an absolute criterion by itself. For country areas, the fast growing $%
u:x\longmapsto \pi x^{2}$ gives a discrepancy from uniformity, whereas the
slow-growing $\log $ does not. However, allowing for exceptions, it is still
a good rule-of-thumb.

\subsection{Continuous r.v.'s}

Our theorems apply on continuous r.v.'s. We now focus on three examples of
such r.v.'s, with the same $u$ as above (except for $\log \circ \log ,$
which is not defined everywhere on $\mathbb{R}_{+}^{\ast }$): the uniform
density on $]0,k]$ $(k>0),$ exponential density, and absolute value of a
normal distribution.

\subsubsection{Uniform r.v.'s}

It is a known fact that a uniform distribution $X_{k}$ on $]0,k]$ $(k>0)$
does not approach classical Benfordness, even as a limit. On every bracket $%
[10^{j-1},10^{j}-1[,$ the leading digit is uniform. Therefore, taking $%
k=10^{j}-1$ leads to a uniform (and not logarithmic) distribution for
leading digits, whatever $j$ might be.

The density $g_{k}$ of $\sqrt{X_{k}}$ is%
\begin{equation*}
g_{k}\left( x\right) =\frac{2x}{k},~x\in \left] 0,\sqrt{k}\right]
\end{equation*}

and $g_{k}\left( x\right) =0$ otherwise. It is an increasing function on $%
]-\infty ,\sqrt{k}]$, decreasing on $[\sqrt{k},+\infty \lbrack $ with
maximum $\frac{2}{\sqrt{k}}\longrightarrow 0$ when $k\longrightarrow \infty .
$ Theorem 2 applies, showing that $X_{k}$ tends toward $\sqrt{.}$%
-Benfordness in law. Now, theorem 3 below proves that $X_{k}$ tends toward $u
$-Benfordness, when $u\left( x\right) =\pi x^{2}.$

\begin{theorem}
If $X$ follows a uniform density on $]0,k],$ $\left\{ \pi X^{2}\right\} $
converges in law toward uniformity on $[0,1[$ when $k\longrightarrow \infty .
$
\end{theorem}

\begin{proof}
Let $X\thicksim U\left( \left] 0,k\right] \right) .$ The p.d.f. $g$ of $%
Y=\pi X^{2}$ is%
\begin{equation*}
g\left( x\right) =\frac{1}{2a\sqrt{x}},~x\in ]0,a^{2}]
\end{equation*}%
where $a=k\sqrt{\pi }.$

The c.d.f. $G$ of $Y$ is then%
\begin{equation*}
G\left( x\right) =\frac{\sqrt{x}}{a},~x\in ]0,a^{2}].
\end{equation*}%
Let now $\delta \in ]0,1[.$ Call $P_{\delta }$ the probability that $\left\{
Y\right\} <\delta $.%
\begin{equation*}
\sum_{j=0}^{\left\lfloor a^{2}-\delta \right\rfloor }G\left( j+\delta
\right) -G\left( j\right) \leq P_{\delta }\leq \sum_{j=0}^{\left\lfloor
a^{2}-\delta \right\rfloor +1}G\left( j+\delta \right) -G\left( j\right)
\end{equation*}%
\begin{equation*}
\frac{1}{a}\sum_{j=0}^{\left\lfloor a^{2}-\delta \right\rfloor }\sqrt{%
j+\delta }-\sqrt{j}\leq P_{\delta }\leq \frac{1}{a}\sum_{j=0}^{\left\lfloor
a^{2}-\delta \right\rfloor +1}\sqrt{j+\delta }-\sqrt{j}
\end{equation*}%
The square-root function being concave,%
\begin{equation*}
\sqrt{j+\delta }-\sqrt{j}\geq \frac{\delta }{2\sqrt{j+\delta }}
\end{equation*}%
and, for any $j>0,$%
\begin{equation*}
\sqrt{j+\delta }-\sqrt{j}\leq \frac{\delta }{2\sqrt{j}}.
\end{equation*}%
Hence,%
\begin{eqnarray*}
\frac{\delta }{2a}\sum_{j=0}^{\left\lfloor a^{2}-\delta \right\rfloor }\frac{%
1}{\sqrt{j+\delta }} &\leq &P_{\delta }\leq \frac{1}{a}\left[ \sqrt{\delta }%
+\sum_{1}^{\left\lfloor a^{2}-\delta \right\rfloor +1}\frac{\delta }{2\sqrt{j%
}}\right] \\
\frac{\delta }{2a}\sum_{j=0}^{\left\lfloor a^{2}-\delta \right\rfloor }\frac{%
1}{\sqrt{j+\delta }} &\leq &P_{\delta }\leq \frac{\sqrt{\delta }}{a}+\frac{%
\delta }{2a}\sum_{1}^{\left\lfloor a^{2}-\delta \right\rfloor +1}\frac{1}{%
\sqrt{j}}
\end{eqnarray*}%
$x\longmapsto \frac{1}{\sqrt{x}}$ being decreasing,%
\begin{equation*}
\sum_{j=0}^{\left\lfloor a^{2}-\delta \right\rfloor }\frac{1}{\sqrt{j+\delta 
}}\geq \int\limits_{\delta }^{\left\lfloor a^{2}-\delta \right\rfloor
+1+\delta }\frac{1}{\sqrt{t}}dt\geq 2\left[ \sqrt{\left\lfloor a^{2}-\delta
\right\rfloor +1+\delta }-\sqrt{\delta }\right]
\end{equation*}%
and%
\begin{equation*}
\sum_{1}^{\left\lfloor a^{2}-\delta \right\rfloor +1}\frac{1}{\sqrt{j}}\leq
\int_{0}^{\left\lfloor a^{2}-\delta \right\rfloor +1}\frac{1}{\sqrt{t}}%
dt\leq 2\left[ \sqrt{\left\lfloor a^{2}-\delta \right\rfloor +1}\right] .
\end{equation*}%
So,%
\begin{equation*}
\frac{\delta }{a}\left[ \sqrt{\left\lfloor a^{2}-\delta \right\rfloor
+1+\delta }-\sqrt{\delta }\right] \leq P_{\delta }\leq \frac{\delta }{a}%
\left[ \sqrt{\left\lfloor a^{2}-\delta \right\rfloor +1}\right] .
\end{equation*}%
As a consequence, for any fixed $\delta $, $\lim_{a\longrightarrow \infty
}\left( P_{\delta }\right) =\delta ,$ and $\left\{ \pi X^{2}\right\} $
converges in law to uniformity on $[0,1[$.
\end{proof}

\subsubsection{Exponential r.v.'s}

Let $X_{\lambda }$ be an exponential r.v. with p.d.f. $f_{\lambda }\left(
x\right) =\lambda \exp \left( -\lambda x\right) $ $\left( x\geq 0,\lambda
>0\right) .$ Engel and Leuenberger [2003] demonstrated that $X_{\lambda }$
tends toward the Benford's law when $\lambda \longrightarrow 0.$

The p.d.f. of $\sqrt{X_{\lambda }}$ is $x\longmapsto 2\lambda x\exp \left(
-\lambda x^{2}\right) ,$ which increases on $\left] 0,\frac{1}{2\lambda }%
\right] $ and then decreases. Its maximum is $\exp \left( -\frac{1}{4\lambda 
}\right) .$ Theorem 2 thus applies, showing that $X_{\lambda }$ is $\sqrt{.}$%
-Benford as a limit when $\lambda \longrightarrow 0.$

Finally, theorem 4 below demonstrates that $X_{\lambda }$ tends toward $u$%
-Benfordness for $u\left( x\right) =\pi x^{2}$ as well.

\begin{theorem}
If $X\thicksim EXP\left( \lambda \right) $ (with p.d.f. $f:x\longmapsto
\lambda \exp \left( -\lambda x\right) ),$ then $Y=\pi X^{2}$ converges
toward uniformity mod 1 when $\lambda \longrightarrow 0.$
\end{theorem}

\begin{proof}
Let $X$ be such a r.v. $Y=\pi X^{2}$ has density $g$ with%
\begin{equation*}
g\left( x\right) =\frac{\mu }{2\sqrt{x}}\exp \left( -\mu \sqrt{x}\right)
,~x\geq 0
\end{equation*}%
where $\mu =\frac{\lambda }{\sqrt{\pi }}.$ The $Y$ c.d.f. $G$ is thus, for
all $x\geq 0$%
\begin{equation*}
G\left( x\right) =1-e^{-\mu \sqrt{x}}.
\end{equation*}%
Let $P_{\delta }$ denote the probability that $\left\{ Y\right\} <\delta ,$
for $\delta \in ]0,1[.$%
\begin{equation*}
P_{\delta }=\sum_{j=0}^{\infty }\left[ e^{-\mu \sqrt{j}}-e^{-\mu \sqrt{%
j+\delta }}\right] 
\end{equation*}%
$x\longmapsto \exp \left( -\mu \sqrt{x}\right) \ $being convex,%
\begin{equation*}
\delta \frac{\mu }{2\sqrt{j+\delta }}e^{-\mu \sqrt{j+\delta }}\leq e^{-\mu 
\sqrt{j}}-e^{-\mu \sqrt{j+\delta }}
\end{equation*}%
for any $j\geq 0,$ and%
\begin{equation*}
e^{-\mu \sqrt{j}}-e^{-\mu \sqrt{j+\delta }}\leq \delta \frac{\mu }{2\sqrt{j}}%
e^{-\mu \sqrt{j}}
\end{equation*}%
for any $j>0.$ Thus%
\begin{equation*}
\delta \sum_{j=0}^{\infty }\frac{\mu }{2\sqrt{j+\delta }}e^{-\mu \sqrt{%
j+\delta }}\leq P_{\delta }\leq 1-e^{-\mu \sqrt{\delta }}+\delta
\sum_{j=1}^{\infty }\frac{\mu }{2\sqrt{j}}e^{-\mu \sqrt{j}}.
\end{equation*}%
$x\longmapsto \frac{1}{\sqrt{x}}\exp \left( -\mu \sqrt{x}\right) $ being
decreasing,%
\begin{eqnarray*}
\delta \sum_{j=0}^{\infty }\frac{\mu }{2\sqrt{j+\delta }}e^{-\mu \sqrt{%
j+\delta }} &\geq &\delta \int_{\sqrt{\delta }}^{\infty }\frac{\mu }{2\sqrt{t%
}}e^{-\mu \sqrt{t}}dt \\
&\geq &\delta \left[ -e^{-\mu \sqrt{t}}\right] _{\sqrt{\delta }}^{\infty } \\
&=&\delta e^{-\mu \sqrt{\delta }},
\end{eqnarray*}%
and%
\begin{eqnarray*}
1-e^{-\mu \sqrt{\delta }}+\delta \sum_{j=1}^{\infty }\frac{\mu }{2\sqrt{j}}%
e^{-\mu \sqrt{j}} &\leq &1-e^{-\mu \sqrt{\delta }}+\delta \int_{0}^{\infty }%
\frac{\mu }{2\sqrt{t}}e^{-\mu \sqrt{t}}dt \\
&\leq &1-e^{-\mu \sqrt{\delta }}+\delta 
\end{eqnarray*}%
The two expressions tend toward $\delta $ when $\mu \longrightarrow 0,$ so
that $P_{\delta }\longrightarrow \delta .$ The proof is complete.
\end{proof}

\subsubsection{Absolute value of a normal distribution}

To test the absolute value of a normal distribution $X$ with mean 0 and
variance $10^{8},$ we picked a sample of 2000 values and used the same
procedure as for real data. It appears, as shown in Table 3, that $X$
significantly departs from $u$-Benfordness with $u=\log $ and $u=\pi .^{2},$
but not with $u=\sqrt{.}.$

\begin{center}
\begin{tabular}{|l|l|l|l|}
\hline
& $\log \left( X\right) $ & $\sqrt{X}$ & $\pi X^{2}$ \\ \hline
$U\left( [0,k[\right) ~k\longrightarrow \infty $ & NO & YES & YES \\ \hline
$EXP\left( \lambda \right) ~\lambda \longrightarrow 0$ & YES & YES & YES \\ 
\hline
$\left\vert \mathcal{N}\left( 0,10^{8}\right) \right\vert $ & $14.49~\left(
.000\right) $ & $0.647~\left( .797\right) $ & $28.726~\left( .000\right) $
\\ \hline
\end{tabular}

Table 3 --- The table displays if uniform distributions, exponential
distributions, and absolute value of a normal distribution, are $u$-Benford
for different functions $u,$ or not. The last line shows the results (and $p$%
-values) of the Kolmogorov-Smirnov tests applied to a 2000-sample. It could
be read as "NO; YES; NO".
\end{center}

As we already noticed, the best shot when one is looking for Benford seems
to be the square-root rather than $\log .$

\section{Discussion}

Random variables exactly conforming the Benford's classical law are rare,
although many do roughly approach the law. Indeed, many explanations have
been proposed for this approximate law to hold so often. These explanations
involve complex characteristics, sometimes directly related to logarithms,
sometimes through multiplicative properties.

Our idea --- formalized in theorem 1 --- is more simple and general. The
fact that real data often are regular and scattered is intuitive. What we
proved is an idea which has been recently expressed by Fewster [2009]:
scatter and regularity are actually \emph{sufficient} condition to
Benfordness.

This fact thus provides a new explanation of Benford's law. Other
explanations, of course, are acceptable as well. But it may be argued that
some of the most popular explanations are in fact corollaries of our
theorem. As we have seen when studying Pareto type II density, mixtures of
distributions may lead to regular and scattered density, to which theorem 1
applies. Thus, we may argue that a mixture of densities is nearly Benford 
\emph{because} it is necessarily scattered and regular. In the same fashion,
multiplications of effects lead to Benford-like densities, but also (as the
multiplicative central-limit theorem states) to regular and scattered
densities.

Apart from the fact that our explanation is simpler and (arguably) more
general, a good argument in its favor is that Benfordness may be generalized
--- unlike log-related explanations. Scale invariance or multiplicative
properties are log-related. But as we have seen, Benfordness is not
dependant on log, and can easily be generalized. Actually, it seems that
square root is a better candidate than log. The historical importance of $%
\log $-Benfordness is of course due to the implications in terms of leading
digits which bears no equivalence with square-root.




















\bibliographystyle{elsarticle-harv}
\bibliography{benford}

\begin{thebibliography}{25}
\expandafter\ifx\csname natexlab\endcsname\relax\def\natexlab#1{#1}\fi
\expandafter\ifx\csname url\endcsname\relax
  \def\url#1{\texttt{#1}}\fi
\expandafter\ifx\csname urlprefix\endcsname\relax\def\urlprefix{URL }\fi

\bibitem[{Benford(1938)}]{BEN38}
Benford, F., 1938. The law of anomalous numbers. Proceedings of the American
  Philosophical Society 78, 551--572.

\bibitem[{Berger et~al.(2004)Berger, Bunimovich, and Hill}]{BER04}
Berger, A., Bunimovich, L., Hill, T., 2004. One-dimensional dynamical systems
  and benford's law. Transactions of the American Mathematical Society 357,
  197--219.

\bibitem[{Burke and Kincanon(1991)}]{BUR91}
Burke, J., Kincanon, E., 1991. Benford's law and physical constants: The
  distribution of initial digits. American Journal of Physics 59, 952.

\bibitem[{Cho and Gaines(2007)}]{CHO07}
Cho, W. K.~T., Gaines, B.~J., 2007. Breaking the (benford) law: Statistical
  fraud detection in campaign finances. The American Statistician 61, 218--223.

\bibitem[{Diaconis(1977)}]{DIA77}
Diaconis, P., 1977. The distribution of leading digits and uniform distribution
  mod 1. The Annals of Probability 5, 72--81.

\bibitem[{Drake and Nigrini(2000)}]{DRA00}
Drake, P.~D., Nigrini, M.~J., 2000. Computer assisted analytical procedures
  using benford's law. Journal of Accounting Education 18, 127--146.

\bibitem[{Engel and Leuenberger(2003)}]{ENG03}
Engel, H.-A., Leuenberger, C., 2003. Benford's law for exponential random
  variables. Statistics and Probability Letters 63, 361--365.

\bibitem[{Fewster(2009)}]{FEW09}
Fewster, R., 2009. A simple explanation of benford's law. The American
  Statistician 63, 26--32.

\bibitem[{Hales et~al.(2008)Hales, Sridharan, Radhakrishnan, Chakravorty, and
  Siha}]{HAL08}
Hales, D.~N., Sridharan, V., Radhakrishnan, A., Chakravorty, S.~S., Siha, S.,
  2008. Testing the accuracy of employee-reported data: An inexpensive
  alternative approach to traditional methods. European Journal of Operational
  Research 189, 583--593.

\bibitem[{Hill(1988)}]{HIL88}
Hill, T., 1988. Random-number guessing and the first-digit phenomenon.
  Psychological Reports 62, 967--971.

\bibitem[{Hill(1995a)}]{HIL95a}
Hill, T., 1995a. Base-invariance implies benford's law. Proceedings of the
  American Mathematical Society 123, 887--895.

\bibitem[{Hill(1995b)}]{HIL95b}
Hill, T., 1995b. A statistical derivation of the significant-digit law.
  Statistical Science 10, 354--363.

\bibitem[{Janvresse and Delarue(2004)}]{JAN04}
Janvresse, E., Delarue, T., 2004. From uniform distributions to benford's law.
  Journal of Applied Probability 41, 1203--1210.

\bibitem[{Jolissaint(2005)}]{JOL05}
Jolissaint, P., 2005. Loi de benford, relations de r\'{e}currence et suites
  \'{e}quidistribu\'{e}es. Elemente der Mathematik 60, 10--18.

\bibitem[{Lolbert(2008)}]{LOL08}
Lolbert, T., 2008. On the non-existence of a general benford's law.
  Mathematical Social Sciences 55, 103--106.

\bibitem[{Mardia and Jupp(2000)}]{MAR00}
Mardia, K.~V., Jupp, P.~E., 2000. Directional statistics. Chichester: Wiley.

\bibitem[{Newcomb(1881)}]{NEW81}
Newcomb, S., 1881. Note on the frequency of use of the different digits in
  natural numbers. American Journal of Mathematics 4, 39--40.

\bibitem[{Paolella(2006)}]{PAO06}
Paolella, M.~S., 2006. Fundamental probability: A computational approach.
  Chichester: Wiley.

\bibitem[{Pietronero et~al.(2001)Pietronero, Tosatti, Tosatti, and
  Vespignani}]{PIE01}
Pietronero, L., Tosatti, E., Tosatti, V., Vespignani, A., 2001. Explaining the
  uneven distribution of numbers in nature: The laws of benford and zipf.
  Physica A 293, 297--304.

\bibitem[{Pinkham(1961)}]{PIN61}
Pinkham, R.~S., 1961. On the distribution of first significant digits. Annals
  of mathematical statistics 32, 1223--1230.

\bibitem[{Raimi(1976)}]{RAI76}
Raimi, R.~A., 1976. The first digit problem. The American Mathematical Monthly
  83, 521--538.

\bibitem[{Scott and Fasli(2001)}]{SCO01}
Scott, P.~D., Fasli, M., 2001. Benford's law: An empirical investigation and a
  novel explanation. Ph.D. thesis, CSM technical report 349, Department of
  Computer Science, University of Essex,
  http://citeseer.ist.psu.edu/709593.html.

\bibitem[{Sehity et~al.(2005)Sehity, Hoelzl, and Kirchler}]{SEH05}
Sehity, T., Hoelzl, E., Kirchler, E., 2005. Price developments after a nominal
  shock: Benford's law and psychological pricing after the euro introduction.
  International Journal of Research in Marketing 22, 471--480.

\bibitem[{Sh\"{u}rger(2008)}]{SHU08}
Sh\"{u}rger, K., 2008. Extensions of black-scholes processes and benford's law.
  Stochastic Processes and their Applications 118, 1219--1243.

\bibitem[{Smith(2007)}]{SMI07}
Smith, S.~W., 2007. The scientist and engineer's guide to digital signal
  processing. http://www.dspguide.com/.

\end{thebibliography}







\end{document}